\documentclass[12pt,twoside]{amsart}
\usepackage{latexsym,amsmath,amsopn,amssymb,amsthm,amsfonts}
\usepackage[T2A]{fontenc}
\usepackage[cp1251]{inputenc}
\usepackage[russian,english]{babel}
\usepackage{graphicx}

\newtheorem{theorem}{Theorem}

\newtheorem{proposition}{Proposition}

\newtheorem{remark}{Remark}

\newtheorem{definition}{Definition}

\newtheorem{lemma}{Lemma}

\newcommand{\Lin}{\operatorname{Lin}}
\newcommand{\const}{\operatorname{const}}
\newcommand{\Ad}{\operatorname{Ad}}
\newcommand{\ad}{\operatorname{ad}}
\newcommand{\diam}{\operatorname{diam}}

\newcommand{\I}{\operatorname{I}}

\newcommand{\Exp}{\operatorname{Exp}}

\newcommand{\Arccos}{\operatorname{arccos}}

\begin{document}
\begin{flushleft}
UDK 519.46 + 514.763 + 512.81 + 519.9 + 517.911
\end{flushleft}
\begin{flushleft}
MSC 22E30, 49J15, 53C17
\end{flushleft}

\title[group $SO(3)$]{(Locally) shortest arcs of special sub-Riemannian metric on the Lie group $SO(3)$}
\author{V.\,N.\,Berestovskii, I.\,A.\,Zubareva,}
\thanks{The first author  was partially supported by the Russian Foundation for Basic Research (Grant 14-01-00068-a) and a grant of the Government of the Russian Federation for the State Support of Scientific Research (Agreement \No 14.B25.31.0029)}
\address{V.N.Berestovskii}
\address{The Sobolev Institute of Mathematics SD RAS, Novosibirsk, Russia}
\email{vberestov@inbox.ru}

\address{I.A.Zubareva}
\address{Omsk Department of the Sobolev Institute of Mathematics, Omsk, Russia}
\email{i\_gribanova@mail.ru}
\maketitle
\maketitle {\small
\begin{quote}
\noindent{\sc Abstract.}
The authors find geodesics, shortest arcs, diameter, cut locus, and conjugate sets for left-invariant sub-Riemannian metric on the Lie group $SO(3)$, under condition that the metric is right-invariant relative to the
Lie subgroup $SO(2)\subset SO(3).$
\end{quote}}

{\small
\begin{quote}
\end{quote}}

\section*{Introduction}

In paper \cite{BerZ} are found exact shapes of sheres of special left-invariant sub-Rieman\-nian metric $d$ on three-dimensional Lie groups: Heisenberg group
$H,$ $SO(3)$ and $SL_2(\mathbb{R}).$

In the last two cases one can give the following natural geometric description of the metric $d.$ The Lie groups $SO(3)$ and $SL_2(\mathbb{R})/\pm E_2$ can be
interpreted as transitive groups of preserving orientation isometries of unit euclidean sphere $S^2$ in three-dimensional Euclidean space and of the Lobachevskii plane $L^2$ with Gaussian curvature $-1$ and hence as spaces $S^2_1$ and $L^2_1$ of unit tangent vectors over these surfaces. The spaces $S^2_1$ and $L^2_1$ admit  Riemannian metric (scalar product) $g_1$ by Sasaki  (see \cite{Sas} or section $1K$ in Besse book \cite{Bes}). In addition, canonical projections
$p:(S^2_1,g_1)\rightarrow S^2$ and $p: (L^2_1,g_1)\rightarrow L^2$ (or, which is equivalent, $p:SO(3)\rightarrow SO(3)/SO(2)$ and
$p: SL_2(\mathbb{R})/\pm E_2\rightarrow  SL_2(\mathbb{R})/SO(2)$ are \textit{Riemannian submersions} \cite{Bes}. The metric $d$ is defined by (totally nonholonomic) left-invariant distribution $D$ on $SO(3)$ and $SL_2(\mathbb{R})/\pm E_2,$ which is orthogonal to fibers of Riemannian submersion $p,$ and restriction of scalar product $g_1$ to $D.$

Moreover, canonical projections
\begin{equation}
\label{subm}
p: (SO(3),d)\rightarrow  S^2,\quad p: SL_2(\mathbb{R})/\pm E_2\rightarrow L^2
\end{equation}
are \textit{submetries} \cite{BG}, natural generalizations of Riemannian submersion. The distribution $D$ on $S^2_1$ and $L^2_1$ is nothing other than the restriction to $S^2_1$ and $L^2_1$ of horizontal distribution of Levi-Civita connection \cite{Bes} for $S^2$ and $L^2.$
Therefore under mentioned identifications of $SO(3)$ and $SL_2(\mathbb{R})/\pm E_2$ with $S^2_1$ and $L^2_1,$ any smooth path $c=c(t), 0\leq t \leq t_1,$ in $SO(3)$ and $SL_2(\mathbb{R})/\pm E_2,$ tangent to the distribution $D,$ is realized as parallel translation of the vector $c(0)\in S^2_1$ and $c(0)\in L^2_1$ along projection $p(c(t)), 0\leq t \leq t_1.$

It follows from here and the Gauss-Bonnet theorem \cite{Pog} for $S^2$ and $L^2$ that canonical projection (to the base of fibration-submersion) of a geodesic in
$(SO(3),d)$ or $(SL_2(\mathbb{R})/\pm E_2,d)$ must be a solution of Dido's
isoperimetric problem (\textit{isoperimetrix}) on the base $S^2$ or $L^2,$ while
a geodesic is a horizontal lift of an isoperimetrix in $S^2$ or $L^2.$ Using this fact, submetries (\ref{subm}) and the suggestion that an isoperimetrix in $S^2$ or
$L^2$ must have constant geodesic curvature, the authors of paper
\cite{BerZ} deduced exact shapes of spheres without searching geodesics and shortest arcs.

In this paper, with the help of mentioned interpretation of geodesics, general methods of paper \cite{Ber1}, and the Gauss-Bonnet theorem for $S^2,$ we find geodesics, shortest arcs, the diameter, cut locus, and conjugate sets in
$(SO(3),d).$ Formulas, analogous to (\ref{sol}) and (\ref{geod}), are obtained in paper \cite{BR}, but we apply other methods and give detailed proofs.

\section{Preliminaries}

Let us recall that the Lie group $Gl(n)=Gl(\mathbb{R}^n)$ consists of all real
$(n\times n)-$matrices $g=(g_{ij}),$ $i,j=1,\dots n$, such that
$\det g \neq 0,$ and the Lie subgroup $Gl_0(n)$ (the connected component
of the unit $e$ in $Gl(n)$) is defined by condition $\det g > 0.$ It is naturally to consider both groups as open submanifolds in $\mathbb{R}^{n^2}$ with coordinates $g_{ij},$ $i,j=1,\dots n.$

Their Lie algebra $\frak{gl}(n)=Gl(n)_e:=Gl_0(n)_e=\mathbb{R}^{n^2}$ is the set of all real $(n\times n)$-matrices with usual structure of vector space and Lie bracket
\begin{equation}
\label{br}
[a,b]= ab - ba;\,\,a,b \in \frak{gl}(n).
\end{equation}
Let \textit{$e_{ij}\in \frak{gl}(n),$ $i,j=1,\dots n$, be a matrix which has 1 in $i$-th row and $j$-th column and 0 in all other places. $\Lin(a,b)$ denotes linear span of vectors $a,b$}. \textit{As an auxiliary tool we shall use standard scalar product $(\cdot,\cdot)$ on the Lie algebra $\frak{gl}(n)=\mathbb{R}^{n^2}$ for $n=3$}. By definition, the Euclidean space $E^n$ is $\mathbb{R}^n$ with standard
scalar product $(x,y)=x^Ty,$ where $x, y\in \mathbb{R}^n$ are regarded
as vector-columns and \textit{$^T$ denotes here and later the transposition
of matrices.}

The Lie group $SO(n)=O(n)\cap Gl_0(n)$ of all orthogonal matrices with the determinant 1 is a connected Lie subgroup in в $Gl_0(n).$ Its Lie algebra $(\frak{so}(n),[\cdot,\cdot])$ is a Lie subalgebra of the Lie algebra
$(\frak{gl}(n),[\cdot,\cdot]),$ consisting of all skew-symmetric matrices.

Let $G$ and $H$ be Lie groups with Lie algebras $\frak{g}$ and $\frak{h}$;
$\phi: G\rightarrow H$ is a Lie groups homomorphism. Then
\begin{equation}
\label{fe}
\phi \circ \exp_{\frak{g}} = \exp_{\frak{h}}\circ d\phi_e,
\end{equation}
moreover,
\begin{equation}
\label{dfe}
d\phi_e : (\frak{g},[\cdot,\cdot]) \rightarrow (\frak{h},[\cdot,\cdot])
\end{equation}
is a Lie algebra homomorphism (see lemma 1.12 in \cite{Hel}). If
$g_0\in G$ then $\I(g_0): G\rightarrow G,$ where $\I(g_0)(g)=g_0gg_0^{-1}$ is inner automorphism of the Lie group $G.$ Consequently,
$\Ad(g_0):=d\I(g_0)_e\in Gl(\frak{g})$ is automorphism of the Lie algebra
$\frak{g}$ and $d\Ad_e(v):=\ad(v):=[v,\cdot]$ for $v\in \frak{g}$ \cite{Hel}. Therefore, on the ground of formula (\ref{fe}),
\begin{equation}
\label{I}
\I(g_0)\circ \exp= \exp \circ \Ad(g_0),
\end{equation}
\begin{equation}
\label{ad}
\Ad(\exp_{\frak{g}}(v))= \exp_{\frak{gl}(\frak{g})}(\ad(v)), v \in \frak{g}.
\end{equation}

In case of left-invariant sub-Riemannian metrics on Lie groups, every geodesic
is a left shift of some geodesic which starts at the unit. Thus later we shall consider only geodesics with unit origin. Theorem 5 in paper \cite{Ber1} implies
the following theorem.

\begin{theorem}
\label{general}
Let $G$ be a connected Lie subgroup of the Lie group $SO(n)\subset Gl_0(n)$
with the Lie algebra $\frak{g},$ $D$ is totally nonholonomic left-invariant distribution on $G,$ a scalar product $\langle \cdot,\cdot\rangle$ on $D(e)$ is proportional to restriction of the scalar product $(\cdot,\cdot)$ (to $D(e)$).
Then parametrized by arclength normal geodesic (i.e. locally shortest arc)
$\gamma=\gamma(t),$ $t\in (-a,a)\subset \mathbb{R},$ $\gamma(0)=e,$ on $(G,d)$ with left-invariant sub-Riemannian metric $d$, defined by distribution $D$ and scalar product $\langle\cdot,\cdot\rangle$ on $D(e),$ satisfies the system of ordinary differential equations
\begin{equation}
\label{dxsp2}
\stackrel{\cdot}\gamma(t)=\gamma(t)u(t),\,\,u(t)\in D(e)\subset \frak{g},\,\,\langle u(t),u(t)\rangle\equiv 1,
\end{equation}
\begin{equation}
\label{vot3}
\stackrel{\cdot}u(t)+\stackrel{\cdot}v(t) = -[u(t),v(t)],
\end{equation}
where $u=u(t),$ $v=v(t)\in \frak{g},$ $(v(t),D(e))\equiv 0,$ $t\in (-a,a)\subset \mathbb{R},$ are some real-analytic vector functions.
\end{theorem}

\section{Geodesics of special left-invariant sub-Riemannian metric on the Lie group  $SO(3)$}

\begin{theorem}
\label{main}
Let be given the basis
\begin{equation}
\label{abc}
a=e_{21}-e_{12},\quad  b=e_{31}-e_{13},\quad c=e_{32}-e_{23}
\end{equation}
of the Lie algebra $\frak{so}(3),$ $D(e)=\Lin(a,b),$ and scalar product $\langle\cdot,\cdot\rangle$ on $D(e)$
with orthonormal basis $a,b.$ Then left-invariant distribution $D$ on the Lie group
$SO(3)$ with given $D(e)$ is totally nonholonomic and the pair
$(D(e),\langle\cdot,\cdot\rangle)$ defines left-invariant sub-Riemannian metric $d$ on $SO(3).$ Moreover, any parametrized by arclength geodesic $\gamma=\gamma(t),$
$t\in \mathbb{R},$ in $(SO(3),d)$ with condition $\gamma(0)=e$ is a product of two 1-parameter subgroups:
\begin{equation}
\label{sol}
\gamma(t)=\exp(t(\cos\phi_0a + \sin \phi_0b +\beta c)) \exp(-t\beta c),
\end{equation}
where $\phi_0,$ $\beta$ are some arbitrary constants.
\end{theorem}

\begin{proof}
It follows from formulae (\ref{br}) and (\ref{abc}) that
\begin{equation}
\label{abca}
[a,b]=c,\quad [b,c]=a,\quad [c,a]= b.
\end{equation}
This implies the first statement of theorem.

It is clear that on $D(e)$
\begin{equation}
\label{scal}
\langle\cdot,\cdot\rangle =\frac{1}{2}(\cdot,\cdot).
\end{equation}
In consequence of theorem 3 in \cite{Ber1} every geodesic on 3-dimensional Lie group with left-invariant sub-Riemannian metric is normal. Then it follows from theorem \ref{general} that one can apply ODE
(\ref{dxsp2}),(\ref{vot3}) to find geodesics $\gamma=\gamma(t), t\in \mathbb{R},$ in $(SO(3),d)$.

It is clear that
\begin{equation}
\label{the}
u(t)=\cos\phi(t)a +\sin\phi(t)b,\quad v(t)=\beta(t)c,
\end{equation}
and the identity (\ref{vot3}) is written in the form
$$-[\cos\phi(t)a +\sin\phi(t)b, \beta(t)c]=
\stackrel{\cdot}\phi(t)(-\sin\phi(t)a +\cos\phi(t)b)+
\stackrel{\cdot}\beta(t)c.$$
In consequence of (\ref{abca}), expression in the left part of equality is equal to
$$\beta(t)(\cos\phi(t)b-\sin\phi(t)a).$$
We get identities $\stackrel{\cdot}\beta(t)=0,$
$\stackrel{\cdot}\phi(t)=\beta(t).$ Hence
\begin{equation}
\label{cond}
\beta=\beta(t)=\const,\quad \phi(t)=\beta t + \phi_0.
\end{equation}

In view of (\ref{dxsp2}), (\ref{the}), and (\ref{cond}), it must be
\begin{equation}
\label{dx}
\stackrel{\cdot}\gamma(t)=\gamma(t)(\cos (\beta t+\phi_0)a + \sin (\beta t+\phi_0)b).
\end{equation}
Let us prove that (\ref{sol}) is a solution of ODE (\ref{dx}).
One can easily deduce from formulae (\ref{abca}) equalities
\begin{equation}
\label{adm}
(\ad (c))=a,\quad (\ad (b))=-b, \quad (\ad (a))=c,
\end{equation}
where $(f)$ denotes the matrix of linear map
$f: \frak{so}(3)\rightarrow \frak{so}(3)$ in the base $a,b,c;$ later $(f)$ is identified with $f$. On the ground of formulae (\ref{ad}), (\ref{adm}), (\ref{cond}), (\ref{the}),
$$\stackrel{\cdot}\gamma(t)=\exp(t(\cos\phi_0a + \sin \phi_0b + \beta c))(\cos\phi_0a + \sin \phi_0b + \beta c)
\exp(-t\beta c)+$$
$$\gamma(t)(-\beta c)=\gamma(t)\exp(t\beta c)(\cos\phi_0a + \sin \phi_0b + \beta c)\exp(-t\beta c)
+\gamma(t)(-\beta c)=$$
$$\gamma(t)\exp(t\beta c)(\cos\phi_0a + \sin \phi_0b)\exp(-t\beta c)+\gamma(t)(\beta c)+\gamma(t)(-\beta c)=$$
$$\gamma(t)\cdot[\Ad(\exp(t\beta c))(\cos\phi_0a + \sin \phi_0b)]=\gamma(t)\cdot[\exp(\ad(t\beta c))(\cos\phi_0a + \sin \phi_0b)]=$$
$$\gamma(t)\cdot[\exp(t\beta(\ad(c)))(\cos\phi_0a + \sin \phi_0b)]=\gamma(t)\cdot[(\exp(t\beta a))(\cos\phi_0a + \sin \phi_0b)]=$$
$$\gamma(t)\cdot (\cos (\beta t+\phi_0)a + \sin (\beta t+\phi_0)b)=\gamma(t)u(t).$$
\end{proof}

\begin{remark}
Both 1-parameter subgroups from formula (\ref{sol}) are nowhere tangent to distribution $D$ for $\beta\neq 0$ so that any their interval has infinite length in metric $d.$
\end{remark}

\begin{remark}
\label{Agr}
On p. 258 in book \cite{AS}, A.A.Agrachev and Yu.L.Sachkov proved that, analogously
to formula (\ref{sol}), every \textit{normal trajectory} (geodesic) of left-invariant sub-Riemannian metric, defined by a distribution with corank 1, on a compact Lie group, starting at the unit, is a product of no more than two 1-parameter subgroups. Let us remind that any geodesic of left-invariant sub-Riemannian metric on 3-dimensional Lie group is normal.
\end{remark}

\begin{proposition}
\label{maint}
Let $\gamma(t),$ $t\in \mathbb{R},$ be geodesic in $(SO_0(2,1),d)$ defined by formula (\ref{sol}). Then for any $t_0\in \mathbb{R},$
\begin{equation}
\label{sol1}
\gamma(t_0)^{-1}\gamma(t)=\exp((t-t_0)(\cos(\beta t_0+\phi_0)a + \sin(\beta t_0+\phi_0)b +\beta c))
\exp(-(t-t_0)\beta c).
\end{equation}
\end{proposition}

\begin{proof}
On the basis of formulae (\ref{I}), (\ref{ad}), (\ref{adm}),
$$\gamma(t_0)^{-1}\gamma(t)=\exp(t_0\beta c)\exp(-t_0(\cos\phi_0a + \sin\phi_0b +\beta c))\cdot$$
$$\exp(t(\cos\phi_0a + \sin \phi_0b +\beta c)) \exp(-t\beta c)=$$
$$\exp(t_0\beta c)\exp((t-t_0)(\cos\phi_0a + \sin \phi_0b +\beta c))\exp(-t_0\beta c) \exp(-(t-t_0)\beta c)=$$
$$[\I(\exp(t_0\beta c))(\exp((t-t_0)(\cos\phi_0a + \sin \phi_0b +\beta c)))]\cdot\exp(-(t-t_0)\beta c)=$$
$$\exp[\Ad(\exp(t_0\beta c)((t-t_0)(\cos\phi_0a + \sin \phi_0b +\beta c))]\cdot\exp(-(t-t_0)\beta c)=$$
$$\exp[\exp(\ad(t_0\beta c))((t-t_0)(\cos\phi_0a + \sin \phi_0b +\beta c))]\cdot\exp(-(t-t_0)\beta c)=$$
$$\exp[\exp(t_0\beta a)((t-t_0)(\cos\phi_0a + \sin \phi_0b +\beta c))]\cdot\exp(-(t-t_0)\beta c)=$$
$$\exp((t-t_0)(\cos(\beta t_0+\phi_0)a + \sin(\beta t_0+\phi_0)b +\beta c))\cdot\exp(-(t-t_0)\beta c).$$
\end{proof}

\begin{remark}
\label{change}
To change a sign of $\beta$ in (\ref{sol}) is the same as to change a sign of $t$
and to change the angle $\phi_0$ by angle $\phi_0\pm \pi.$
\end{remark}

\begin{remark}
\label{Ad}
For any matrix $B\in SO(2)=\exp(\mathbb{R}c),$ the map $l_B\circ r_{B^{-1}}$,
where $l_B$ is multiplication from the left by $B$, $r_{B^{-1}}$ is multiplication from the right by $B^{-1}$, is simultaneously automorphism $\Ad B$ of the Lie algebra
$(\frak{so}(3),[\cdot,\cdot]),$ preserving $\langle\cdot,\cdot\rangle,$ and automorphism of the Lie group $SO(3),$ preserving distribution $D$ and metric
$d.$ In particular in view of (\ref{ad}),
(\ref{adm})
$$\Ad B (a+\beta c)=\exp(\phi_0 a)(a+\beta c)=\cos\phi_0a + \sin \phi_0b + \beta c,$$ if
\begin{equation}
\label{AA}
B=\exp(\phi_0c)=\left(\begin{array}{ccc}
1 & 0 & 0\\
0 & \cos \phi_0 & -\sin \phi_0 \\
0 &   \sin \phi_0  & \cos \phi_0
\end{array}\right).
\end{equation}
\end{remark}

\begin{lemma}
\label{ee}
\begin{equation}\label{expx}
\exp(t(a+\beta c))=\I(\exp(-\xi b))(\exp(t\sqrt{1+\beta^2}a)),
\end{equation}
where
\begin{equation}
\label{cosi}
\cos\xi=\frac{1}{\sqrt{1+\beta^2}},\quad \sin\xi=\frac{\beta}{\sqrt{1+\beta^2}}.
\end{equation}
\end{lemma}

\begin{proof}
Taking into account (\ref{cosi}), (\ref{adm}), (\ref{ad}), we get
$$t(a+\beta c)=(t\sqrt{1+\beta^2}(\cos\xi\cdot a + \sin\xi\cdot c))=(\exp(\xi b))(t\sqrt{1+\beta^2}a)=$$
$$(\exp(\ad(-\xi b)))(t\sqrt{1+\beta^2}a)=\Ad(\exp(-\xi b))(t\sqrt{1+\beta^2}a).$$
Now in consequence of obtained equalities and (\ref{I}),
$$\exp(t(a+\beta c))=\exp(\Ad(-\xi b)(t\sqrt{1+\beta^2}a))=\I(\exp(-\xi b))(\exp(t\sqrt{1+\beta^2}a)).$$
\end{proof}

\begin{theorem}
\label{matrexp}
The geodesic $\gamma=\gamma(t)$ of left-invariant sub-Riemannian metric $d$ on the Lie group $SO(3),$ defined by formula (\ref{sol}), is equal to
\begin{equation}
\label{geod}
\tiny{\left(\begin{array}{ccc}
1-n & -m\cos{(\beta t+\phi_0)}-\beta n\sin{(\beta t+\phi_0)} & -m\sin{(\beta t+\phi_0)}+\beta n\cos{(\beta t+\phi_0)} \\
m\cos{\phi_0}-\beta n\sin{\phi_0} & (1-\beta^2n)\cos{\beta t}+\beta m\sin{\beta t}-n\cos{(\beta t+\phi_0)}\cos{\phi_0} &
(1-\beta^2n)\sin{\beta t}-\beta m\cos{\beta t}-n\sin{(\beta t+\phi_0)}\cos{\phi_0} \\
m\sin{\phi_0}+\beta n\cos{\phi_0} & \beta m\cos{\beta t}-(1-\beta^2n)\sin{\beta t}-n\cos{(\beta t+\phi_0)}\sin{\phi_0} &
(1-\beta^2n)\cos{\beta t}+\beta m\sin{\beta t}-n\sin{(\beta t+\phi_0)}\sin{\phi_0}
\end{array}\right),}
\end{equation}
where
\begin{equation}\label{mn}
m=\frac{\sin{(t\sqrt{1+\beta^2})}}{\sqrt{1+\beta^2}},\,\,n=\frac{1-\cos{(t\sqrt{1+\beta^2})}}{1+\beta^2}.
\end{equation}
\end{theorem}

\begin{proof}
Let $\phi_0=0$. Then (\ref{sol}) takes the form
$$\gamma(t)\mid_{\phi_0=0}=\exp{(t(a+\beta c))}\exp{(-t\beta c)}.$$
Using lemma~\ref{ee}, (\ref{mn}) and carrying out routine calculations, we get
$$\exp{(t(a+\beta c))}=$$
$$\frac{1}{1+\beta^2}\left(\begin{array}{ccc}
1 & 0 & \beta \\
0 & \sqrt{1+\beta^2} & 0 \\
-\beta & 0 & 1
\end{array}\right)
\left(\begin{array}{ccc}
\cos t\sqrt{1+\beta^2} & -\sin t\sqrt{1+\beta^2} & 0 \\
\sin t\sqrt{1+\beta^2} & \cos t\sqrt{1+\beta^2} & 0 \\
0 & 0 & 1
\end{array}\right)\times$$
$$\left(\begin{array}{ccc}
1 & 0 & -\beta \\
0 & \sqrt{1+\beta^2} & 0 \\
\beta & 0 & 1
\end{array}\right)= \left(\begin{array}{ccc}
1-n & -m & n\beta \\
m & 1-n(1+\beta^2) & -m\beta \\
n\beta & m\beta & 1-n\beta^2
\end{array}\right).$$
Now, using (\ref{sol}) and (\ref{AA}) for $\phi_0=-\beta t$, we get
$$\gamma(t)\mid_{\phi_0=0}=
\left(\begin{array}{ccc}
1-n & -m & n\beta \\
m & 1-n(1+\beta^2) & -m\beta \\
n\beta & m\beta & 1-n\beta^2
\end{array}\right)\cdot\left(\begin{array}{ccc}
1 & 0 & 0 \\
0 & \cos \beta t & \sin \beta t\\
0 & -\sin \beta t & \cos \beta t
\end{array}\right)=$$
$$\left(\begin{array}{ccc}
1-n & -m\cos{\beta t}-n\beta\sin{\beta t} & \beta n\cos{\beta t}-m\sin{\beta t} \\
m & (1-n(1+\beta^2))\cos{\beta t}+m\beta\sin{\beta t} &
(1-n(1+\beta^2))\sin{\beta t}-m\beta\cos{\beta t} \\
\beta n & m\beta\cos{\beta t}-(1-\beta^2n)\sin{\beta t} &
(1-\beta^2n)\cos{\beta t}+m\beta\sin{\beta t}
\end{array}\right).$$
By (\ref{AA}), matrices $B=\exp(\phi_0)$ and $\exp{(-t\beta c)}$ commute. It follows from here and from remark \ref{Ad} that
\begin{equation}
\label{gb}
\gamma(t)=
B\cdot\gamma(t)\mid_{\phi_0=0}\cdot B^{-1}.
\end{equation}
Substitution of formula (\ref{AA}) into the last equality finishes the proof.
\end{proof}

\section{Shortest arcs on the Lie group $(SO(3),d)$}
\label{geom}

The group $SO(3)$ is realized as the group of all preserving orientation isometries
$$v\rightarrow gv; \quad g \in SO(3), v\in S^2$$
of unit sphere $S^2\subset \mathbb{R}^3,$ whose elements $v$ are regarded as vector-columns. It is not difficult to check that Lie subgroup
$$SO(2):= \{\exp s c, s\in \mathbb{R}\}\subset SO(3)$$
is the stabilizer of vector $v_0=(1,0,0)^T=e_1\in S^2$ with respect to this action.
Moreover the group $SO(2)$ acts (simply) transitively by rotations on unit circle
$S^1:=S^2\cap e_1^{\perp}\subset S^2.$

Therefore $S^2$ is naturally identified with quotient homogeneous space
$SO(3)/SO(2)$ and the group $SO(3)$ itself is diffeomorphic to the space $S^2_1$ of all unit tangent vectors to $S^2.$ Namely, every element $g\in SO(3)$ corresponds to $ge'_2,$ where $e'_2$ is usual parallel translation of vector $e_2$ to point
$e_1$. Moreover, in consequence of introduction,

1) Any segment of a smooth path $c=c(t)$ in $(SO(3),d),$ tangent to distribution
$D,$ has the same length as its image relative to canonical projection
\begin{equation}
\label{p}
p:g\in SO(3)\rightarrow ge_1\in S^2;
\end{equation}

2) under indicated identification of $SO(3)$ with $S^2_1,$ any path
$c=c(t), 0\leq t \leq t_1,$ tangent to distribution $D,$ is realized as  parallel vector field in $S^2$ along $p(c(t)), 0\leq t \leq t_1,$ with initial unit tangent vector $c(0)\in S^2_1;$

3) By the Gauss-Bonnet theorem \cite{Pog}, under parallel translation in
$S^2$ of non-zero tangent vector along a contour, bounding a region in
$S^2$ with area $S < 2\pi,$ the vector turns in the direction of bypass by
the angle $S.$

Let us use statements 1) --- 3) to find shortest arcs in
$(SO(3),d)$. In consequence of proposition \ref{maint}, remark \ref{Ad}, and  left invariance of the metric $d$, it is sufficient to investigate segments of geodesics of the form
\begin{equation}
\label{ge}
\gamma(t)=\exp(t(a+\beta c))\exp(-t\beta c),\quad 0\leq t\leq t_1,
\end{equation}
and their projections
\begin{equation}
\label{pr}
x(t):= p(\gamma(t))=\gamma(t)\cdot e_1=\gamma(t)\cdot (1,0,0)^T=(1-n,m,\beta n)^T,\quad 0\leq t\leq t_1,
\end{equation}
to the sphere $S^2,$ where $m$, $n$ are defined by formulae (\ref{mn}) (we
used formula (\ref{geod}) for  $\phi_0=0$).

Since the second factor in (\ref{ge}) lies in $SO(2),$ then orbits (\ref{pr}) coincide with segments of orbits of 1-parameter subgroup $y(t)=\exp(t(a +\beta c)),$ $t\in \mathbb{R}.$

It is not difficult to calculate that $\pm (1/\sqrt{1+\beta^2})(\beta,0,1)^T\in S^2$ are unit eigenvectors of matrix $a+\beta c$ with respect to zero eigenvalue. Consequenly, 1-parameter subgroup $y(t),$ $t\in \mathbb{R},$ preserves these vectors. Scalar products of these vectors with $e_1$ are equal to $\pm (\beta/\sqrt{1+\beta^2}).$ Then spherical distance from the point $e_1$ to the axis of these vectors
is equal to
\begin{equation}
\label{r}
r=\mbox{arccos}(|\beta|/\sqrt{1+\beta^2}) \leq \pi/2.
\end{equation}
Therefore the orbit $\{\gamma(t)e_1=y(t)e_1\}$ is spherical circle of radius
$r < \pi/2$ with unique center $(1/\sqrt{1+\beta^2})(\beta,0,1)^T,$ if
$\beta\neq 0.$ It is not difficult to see that if $\beta > 0,$ then
in consequence of theorem \ref{matrexp}, curve (\ref{pr}) for
$t_1=2\pi/\sqrt{1+\beta^2}$ goes around this circle, bounding lesser region
$\Psi$ of $S^2$ with this center inside it, one times, \textit{leaving the region  $\Psi$ from the left}.

Let us formulate the Gauss-Bonnet theorem \cite{Pog}. Let $M$ be
two-dimensional oriented manifold with Riemannian metric $ds^2,$ $\Phi$ is a region in $M,$ homeomorphic to disc and bounded by closed piece-wise regular curve
$\gamma$ with regular links $\gamma_1,\dots, \gamma_n,$ forming angles
$\alpha_1, \dots, \alpha_n$ from the side of region $\Phi.$ Direction on the curve
$\gamma$ is given so that the region $\Phi$ is situated from the right under bypass of the curve in this direction. Then

\begin{theorem}
\label{gauss}
\begin{equation}
\label{ga}
\sum_{k=1}^n \int_{\gamma_k}\kappa ds + \sum_{k=1}^n (\pi-\alpha_k) = 2\pi -\int\int_{\Phi}Kd\sigma,
\end{equation}
where $\kappa$ is geodesic curvature at points of links of the curve, $K$ is
Gaussian (sectional) curvature of the surface $(M,ds^2),$ and integration in the
right part of equality is taken by area element of the region $\Phi.$

In particular, if $\gamma$ is a regular curve, then
\begin{equation}
\label{gau}
\int_{\gamma}\kappa ds = 2\pi -\int\int_{\Phi}Kd\sigma.
\end{equation}
\end{theorem}

\begin{proposition}
\label{kap}
Geodesic curvature of curve (\ref{pr}) for $\beta > 0$ is equal to $-|\beta|.$
\end{proposition}

\begin{proof}
In consequence of what has been said, applying equality (\ref{gau}) to circle (\ref{pr}) for $\beta > 0$ and $t_1=2\pi/\sqrt{1+\beta^2},$ one needs to take region $\Phi=S^2\smallsetminus \overline{\Psi}$ in $S^2$ and $K=1.$
Then the left part of (\ref{gau}) is equal to $\kappa t_1.$ For the right part,
we need area $\sigma(\Phi).$

It is known that in $S^2$
\begin{equation}
\label{l1}
l(r,\alpha)= \alpha \sin r,
\end{equation}
\begin{equation}
\label{S}
S(r,\alpha)=\int_0^r \alpha \sin s ds = \alpha \ch s|_0^r= \alpha (1-\cos r),
\end{equation}
where $l(r,\alpha)$ is the length of arc of circle with radius $r$ and central
angle $\alpha \leq 2\pi,$ and $S(r,\alpha)$ is area of corresponding sector. Then in consequense of (\ref{r}),
$$\sigma(\Psi)=2\pi\left(1-\frac{|\beta|}{\sqrt{1+\beta^2}}\right),$$
$$\sigma(\Phi)=4\pi-\sigma(\Psi)= 2\pi\left(1+\frac{|\beta|}{\sqrt{1+\beta^2}}\right),$$
$$\frac{2\pi\kappa}{\sqrt{1+\beta^2}}= 2\pi - \sigma(\Phi)= -2\pi\frac{|\beta|}{\sqrt{1+\beta^2}},
\quad \kappa=-|\beta|.$$
\end{proof}

\begin{proposition}
\label{area}
Let us assume that projection (\ref{pr}) of geodesic segment (\ref{ge}), where
$\beta\neq 0,$ has no self-intersection, i.e.
$0\leq t_1< 2\pi/\sqrt{1+\beta^2}$, $S(t_1)=S(t_1,\beta)$ is area of lesser
curvilinear digon $P$ in $S^2$, bounded by segment (\ref{pr}) and shortest
segment $[x(0)x(t_1)]$ of a length $r=r(t_1)$ in $S^2,$
$\psi=\psi(t_1,\beta)$ is interior angle of the digon $P$. Then
\begin{equation}
\label{exist}
r = \Arccos((1-n)(t_1)), \quad r'(t_1)=\cos\psi = \frac{m}{\sqrt{n(2-n)}},\quad S(t_1)=2\psi- |\beta|t_1.
\end{equation}
Moreover $S'(t_1)>0,$ if $t_1>0;$ $0< \psi \leq \pi/2,$ if $0 < t_1 \leq \pi/\sqrt{1+\beta^2}$, and $\pi/2 < \psi < \pi,$ if $\pi/\sqrt{1+\beta^2} < t_1 \leq 2\pi/\sqrt{1+\beta^2}.$
\end{proposition}

\begin{proof}
The first equality in  (\ref{exist}) is a corollary of (\ref{pr}) and known formula for distance in spherical geometry, the second one is a well-known statement of Riemannian geometry (on existence of strong angle), the third equality is result
of differentiation of first equality in (\ref{exist}). Inequalities for the angle are evident. In consequence of remark \ref{change} one can assume that $\beta > 0.$
Segment $[x(0)x(t_1)]$ has geodesic curvature $0.$ Then, with taking into account
$\Phi=S^2\smallsetminus \overline{P}$ and proposition \ref{kap},  equation (\ref{ga}) is written in the form
$$-|\beta|t_1 + (2\pi - (4\pi -2\psi))=2\pi - (4\pi - S(t_1)).$$
Consequently, $S(t_1)= 2\psi - |\beta|t_1.$  From here and (\ref{S}) follow
relations
$$S'(t_1)=2\psi'(t_1)-|\beta|=(1-\cos r)\psi'(t_1),$$
\begin{equation}
\label{dpsi}
\psi'(t_1)=\frac{|\beta|}{1 + \cos r}=\frac{|\beta|}{2-n},
\end{equation}
\begin{equation}
\label{der}
S'(t_1)=|\beta|\left(\frac{2}{2-n}-1\right)(t_1)>0, \quad 0 < t_1 < \frac{2\pi}{\sqrt{1+\beta^2}}.
\end{equation}
\end{proof}

\begin{lemma}
\label{zer}
If $\beta=0$ and $t_1=\pi,$ then (\ref{ge}) is noncontinuable shortest arc.
\end{lemma}

\begin{proof}
In this case $\gamma(t)=\exp(ta)$. Then $\gamma(2\pi)=e$ and, consequently,
$\gamma(\pi)=\gamma(-\pi).$ Therefore geodesic segment $\gamma(t)$,
$0\leq t\leq t_2$, is not shortest arc for $t_2> t_1=\pi$. On the other hand, canonical projection $p: (SO(3),d)\rightarrow S^2$ (see (\ref{subm}) и (\ref{p}))
is a submetry, moreover
$$\gamma(\pi)=-(e_{11}+e_{22})+e_{33},\quad p(\gamma(\pi))=\gamma(\pi)e_1=-e_1,$$
i.e. path $p(\gamma(t)),$ $0\leq t \leq \pi,$ is shortest connection in $S^2$ of diametrally opposite points $e_1$ and $-e_1$. Then (\ref{ge}) is noncontinuable shortest arc.
\end{proof}

\begin{proposition}
\label{zer1}
1)If $\beta\neq 0$ then geodesic segment (\ref{ge}) is noncontinuable shortest arc
when its projection (\ref{pr}) is a) one time passing circle $C$ bounding disc
with area $S(t_1)\leq \pi$ or b) curve without self-intersections bounding together with the shortest arc $[x(0)x(t_1)]$ in $S^2$ digon $P$ in $S^2$ with area
$S(t_1) =\pi$.

2) For every $\beta\neq 0$ there is unique $t_1 > 0$ such that one of conditions
a) or b) is satisfied; a) is satisfied only if $|\beta|\geq 1/\sqrt{3}.$
\end{proposition}

\begin{proof}
1) a) It is clear that $\gamma(t_1)\in SO(2).$ Then in consequence of
remark  \ref{Ad}, segment of geodesic
(\ref{sol}) for the same $\beta$ and any $\phi_0$ under $t\in [0,t_1]$ joins the same points as (\ref{ge}). Consequently every continuation of the segment (\ref{ge}) is not a shortest arc.

Let us suppose that there exists a shortest arc $\gamma_2(t),$
$0\leq t\leq t_2 < t_1,$ in $(SO(3),d)$ which joins points $\gamma(0)=e$ and
$\gamma(t_1).$ Then projection $x_2(t)=p(\gamma_2(t)),$ $0\leq t\leq t_2,$ is one time passing circle $C_2$ in $S^2$ with length $t_2 < t_1$ and therefore bounds
a disc with area $S(t_2)< S(t_1)\leq \pi.$ Consequently on the ground of the Gauss-Bonnet theorem, results of parallel translations of nonzero vectors along $C$ and $C_2$ in $S^2$ are different. Then $\gamma_2(t_2)\neq \gamma(t_1)$
in view of geometric interpretation of geodesics in $(SO(3),d),$ given in introduction, a contradiction.

b) Let $P'$ be a digon, symmetric to the digon $P$ relative to segment
$[x(0)x(t_1)].$ Since $S(t_1)=\pi$ then by the Gauss-Bonnet theorem, results of parallel translations in $S^2$ of tangent vectors along closed paths, bounding $P$ and $P',$ are equal. Therefore on the ground of remarks \ref{change}, \ref{Ad} and geometric interpretation of geodesics in $(SO(3),d),$ given in introduction,
a curve in $S^2,$ symmetric to the projection (\ref{pr}) of segment (\ref{ge}) relative to segment $[x(0)x(t_1)],$ is presented in the form $p(\gamma_1(t)),$
$0\leq t \leq t_1,$ where $\gamma_1$ is a geodesic in $(SO(3),d)$ such that
$\gamma_1(0)=\gamma(0),$ $\gamma_1(t_1)=\gamma(t_1).$ Consequently every continuation of the segment (\ref{ge}) is not a shortest arc.

Let us suppose that there is a shortest arc $\gamma_2(t),$ $0\leq t\leq t_2 < t_1,$ in $(SO(3),d),$ joining points $\gamma(0)=e$ and $\gamma(t_1).$ Then
in consequence of remarks \ref{change} and \ref{Ad} we can assume that curves
(\ref{pr}) and $x_2(t)=p(\gamma_2(t)),$ $0\leq t\leq t_2,$ lie on the one side
of the shortest arc $[x(0)x(t_1)]$ and join ends of this shortest arc. Consequently  the digon $P$ and digon $P_2,$ bounded by the shortest arc $[x(0)x(t_1)]$ and the curve $x_2(t),$ $0\leq  t \leq t_2,$ are convex, moreover intersection of their boundaries is the shortest arc
$[x(0)x(t_1)],$  because $t_2 < t_1.$ Therefore in view of last inequality the curve
$x_2(t),$ $0< t < t_2,$ lies inside $P$ and $S(t_2)< S(t_1)=\pi,$ where $S(t_2)$ is area of the digon $P_2.$ Consequently on the ground of the Gauss-Bonnet theorem,
results of parallel translations of nonzero tangent vectors along boundaries of $P$ and $P_2$ in $S^2$ are different. Then $\gamma_2(t_2)\neq \gamma(t_1)$
in view of geometric interpretation of geodesics in $(SO(3),d),$ given in introduction, a contradiction.

2) On the ground of last equality in (\ref{exist}), the condition a) is fulfilled only if $t_1=2\pi/\sqrt{1+\beta^2},$ $\psi(t_1)=\pi$ and
$$\quad S\left(\frac{2\pi}{\sqrt{1 + \beta^2}}\right)= 2\pi - |\beta|\frac{2\pi}{\sqrt{1 + \beta^2}} \leq \pi
\Leftrightarrow |\beta|\geq \frac{1}{\sqrt{3}}.$$
If $0< |\beta|< \frac{1}{\sqrt{3}},$ then in consequence of proposition \ref{area} there exists unique $t_1> 0$ for which the condition b) is satisfied.
\end{proof}

Later for every number $\beta\neq 0$ we shall find a number $t_1=t_1(\beta),$ satisfying conditions of proposition \ref{zer1}.

I) If $|\beta|\geq \frac{1}{\sqrt{3}},$ then $t_1=2\pi/\sqrt{1 + \beta^2}.$

II) If $0 < |\beta| < \frac{1}{\sqrt{3}},$ then
$$S(2\pi/\sqrt{1 + \beta^2}) < 2\pi  \Rightarrow S(\pi/\sqrt{1 + \beta^2}) < \pi,$$
\begin{equation}
\label{to}
S(t_1)=\pi \Rightarrow \pi/\sqrt{1+\beta^2} < t_1 < 2\pi/\sqrt{1+\beta^2}.
\end{equation}
Therefore in consequence of proposition \ref{area}, $\pi/2 < \psi(t_1) < \pi$ and
\begin{equation}
\label{equal}
S(t_1)=\pi \Leftrightarrow \quad 2\psi - |\beta|t_1 = \pi \Leftrightarrow \frac{|\beta|t_1}{2} = \psi - \pi/2.
\end{equation}
In consequence of (\ref{exist}) and (\ref{mn}),
$$\cos\psi= \frac{m}{\sqrt{n(2-n)}}=\frac{\sqrt{1+\beta^2}\sin(t_1\sqrt{1+\beta^2})}
{\sqrt{(1-\cos(t_1\sqrt{1+\beta^2}))(1+\cos(t_1\sqrt{1+\beta^2})+2\beta^2)}}=$$
$$\frac{\sqrt{1+\beta^2}\cos(t_1\sqrt{1+\beta^2}/2)}
{\sqrt{\cos^2(t_1\sqrt{1+\beta^2}/2)+\beta^2}}.$$
We get from here, (\ref{equal}), inequalities for $t_1$ and $\psi$ that
$$\sin\psi = \sqrt{1-\cos^2\psi}= \frac{|\beta|\sin(t_1\sqrt{1+\beta^2}/2)}
{\sqrt{\cos^2(t_1\sqrt{1+\beta^2}/2)+\beta^2}},$$
\begin{equation}
\label{sin}
\sin\left(\frac{|\beta|t_1}{2}\right)=\sin\left(\psi-\frac{\pi}{2}\right)=
-\cos\psi=\frac{-\sqrt{1+\beta^2}\cos(t_1\sqrt{1+\beta^2}/2)}
{\sqrt{\cos^2(t_1\sqrt{1+\beta^2}/2)+\beta^2}},
\end{equation}
\begin{equation}
\label{cos}
\cos\left(\frac{|\beta|t_1}{2}\right)=\cos\left(\psi-\frac{\pi}{2}\right)=
\sin\psi=\frac{|\beta|\sin(t_1\sqrt{1+\beta^2}/2)}
{\sqrt{\cos^2(t_1\sqrt{1+\beta^2}/2)+\beta^2}},
\end{equation}
\begin{equation}
\label{bt}
0 < |\beta|t_1 < \pi.
\end{equation}

\begin{theorem}
\label{monot}
Conditions a),b) of proposition \ref{zer1} define a continuous function
$t_1=t_1(|\beta|),$ increasing under $0\leq |\beta| \leq 1/\sqrt{3}$ and decreasing under $1/\sqrt{3} \leq |\beta| < + \infty$.
\end{theorem}

\begin{proof}
The second statement is evident. The first statement is true, because
$dt_1/d|\beta|>0$ under $0< |\beta|< 1/\sqrt{3}$ in consequence of (\ref{equal}), (\ref{dpsi}), (\ref{mn}), (\ref{to}):
$$t_1+|\beta|\frac{dt_1}{d|\beta|}=2\psi'(t_1)\cdot\frac{dt_1}{d|\beta|}=\frac{2|\beta|}{2-n}\cdot\frac{dt_1}{d|\beta|},$$
$$t_1=\frac{|\beta|n}{2-n}\cdot\frac{dt_1}{d|\beta|}=
\frac{|\beta|\sin^2(t_1\sqrt{1+\beta^2}/2)}{\beta^2+\cos^2(t_1\sqrt{1+\beta^2}/2)}\cdot\frac{dt_1}{d|\beta|}.$$
\end{proof}

\begin{theorem}
\label{diam} $$\diam (SO(3),d)=\pi\sqrt{3}.$$
\end{theorem}

\begin{proof}
It follows from theorem \ref{monot} that maximal length of shortest arc is attained under $\beta^2=1/3$ and it is equal to $\pi\sqrt{3}.$ This implies
needed statement.
\end{proof}

\begin{remark}
\label{diamo}
Statement of theorem \ref{diam} is a particular case of the first statement of theorem 2 from paper \cite{BerZ}.
\end{remark}

\section{Cut locus and conjugate sets in $(SO(3),d)$}
\label{cut}

Unlike Riemannian manifolds, exponential map $\Exp$ and its restriction $\Exp_x$
for sub-Riemannian manifold $(M,d)$ without abnormal  geodesics (as in the case of $(SO(3),d)$) are defined not on $TM$ and $T_xM$ but only on $D$ and $D(x)$,
where $D$ is distribution on $M,$ taking part in definition of $d.$ Otherwise
cut locus and conjugate sets for such sub-Riemannian manifold are defined in the same way as for Riemannian one \cite{GKM}.

\begin{definition}
\label{cutl}
Cut locus $C(x)$ (respectively conjugate set $S(x)$) for a point $x$ in sub-Riemannian manifolds $M$ (without abnormal geodesics) is the set of ends
of all noncontinuable beyond its ends shortest arcs starting at the point $x$ (respectively, image of the set of critical points of the map $\Exp_x$
with respect to $\Exp_x$).
\end{definition}

\begin{theorem}
\label{clcl}
For every element $g\in (SO(3),d),$ $C(g)=gC(e)$ and $S(g)=gS(e).$ Moreover $S(g)\subset C(g),$
\begin{equation}
\label{cutle}
C(e)=\{\gamma_{\beta}(t_1(\beta)): \beta\in \mathbb{R}\},
\end{equation}
\begin{equation}
\label{conle}
S(e)=\{\gamma_{\beta}(t_1(\beta)): \beta^2\geq 1/3\}= SO(2)\smallsetminus \{e\};
\end{equation}
$S(e)$ is diffeomorphic to $\mathbb{R};$
\begin{equation}
\label{os}
\overline{S(e)}=S(e)\cup \{e\}=SO(2),
\end{equation}
$\overline{S(e)}$ is diffeomorphic to circle $S^1$;
\begin{equation}
\label{ocs}
\overline{C(e)\smallsetminus S(e)}= (C(e)\smallsetminus S(e))\cup \left\{\gamma_{\beta}(t_1(\beta))=
\gamma_{-\beta}(t_1(-\beta)): \beta = \frac{1}{\sqrt{3}}\right\},
\end{equation}
$\overline{C(e)\smallsetminus S(e)}$ is diffeomorphic to real projective plane
$RP^2;$  $C(e)$ is homeomorphic to $RP^2\cup \mathbb{R},$ where $RP^2\cap \mathbb{R}$ is one-point set;
$\overline{C(e)}$ is homeomorphic to $RP^2\cup S^1,$
where $RP^2\cap S^1$ is one-point set.
\end{theorem}

\begin{proof}
First statement is a corollary of left invariance of the metric $d$ on $SO(3).$ Inclusion $S(g)\subset C(g),$ formulae
(\ref{cutle}), (\ref{conle}), equality in brace from (\ref{ocs}), and diffeomorphism $S(e)\cong \mathbb{R}$ are corollaries from the proof of
proposition~\ref{zer1} and remark \ref{Ad}. Formula (\ref{os}) and diffeomorphism
$\overline{S(e)}\cong S^1$ follow from formula (\ref{conle}). Equality (\ref{ocs}) follows from formulae (\ref{cutle}), (\ref{conle});
$\overline{C(e)\smallsetminus S(e)}\cong RP^2$ follows from equalities
$\gamma_{(\beta,\phi_0)}(t_1(\beta))=\gamma_{(-\beta,-\beta t_1+\phi_0 + \pi)}(t_1(-\beta))$ при $\beta^2\leq 1/3.$
Now it is not difficult to prove remaining statements.
\end{proof}

\begin{remark}
It follows from (\ref{os}) and equalities $C(g)=gC(e),$ $S(g)=gS(e)$ that
$g\in gSO(2)=\overline{S(g)}\subset \overline{C(g)}$ for all $g\in SO(3),$
while $x\notin \overline{C(x)}$ and $x\notin \overline{S(x)}$ for any point
$x$ of arbitrary smooth Riemannian manifold. This constitutes radical difference
of Riemannian and sub-Riemannian manifolds.
\end{remark}

\end{document}